\documentclass{amsart}
\usepackage{amsfonts,amsmath,color,graphicx,calligra,mathrsfs,tikz,tikz-cd,mathtools}
\usepackage{amsthm,amssymb}

\usetikzlibrary{matrix,calc,arrows}
\tikzcdset{scale cd/.style={every label/.append style={scale=#1},
		cells={nodes={scale=#1}}}}

\usetikzlibrary{decorations.markings}
\tikzset{negated/.style={
		decoration={markings,
			mark= at position 0.5 with {
				\node[transform shape] (tempnode) {$\backslash$};
			}
		},
		postaction={decorate}
	}
}

\usepackage[english]{babel}
\usepackage{comment}
\usepackage{hyperref}
\usepackage{cleveref}

\newtheorem{thm}{Theorem}[section]

\newtheorem{lem}[thm]{Lemma}
\newtheorem{prop}[thm]{Proposition}
\theoremstyle{definition}
\newtheorem{defn}[thm]{Definition}
\theoremstyle{remark}
\newtheorem{rem}[thm]{Remark}
\theoremstyle{definition}

\theoremstyle{definition}

\theoremstyle{remark}

\numberwithin{equation}{section}

\newcommand{\locala}{(A,\mathfrak{m},k)}
\newcommand{\localb}{(B,\mathfrak{n},l)}

\newcommand{\Tor}{\mathrm{Tor}}

\newcommand{\HH}{\mathrm{H}}

\newcommand{\eqdef}{\mathrel{\mathop:}=}

\newcommand{\pdim}{\mathrm{pd}}

\newcommand{\cidim}{\mathrm{CI}\text{-}\mathrm{dim}}
\newcommand{\xMapsto}[2][]{\ext@arrow 0599{\Mapstofill@}{#1}{#2}}

\def\Mapstofill@{\arrowfill@{\Mapstochar\Relbar}\Relbar\Rightarrow}

\makeatother

\NewDocumentEnvironment{property}{m o m}{
	\par\addvspace{\topsep}
	\noindent\hypertarget{#3}{\textbf{(#3): #1}}
	\IfValueT{#2}{\space(#2)}
	\phantomsection
	\label{#3}
	\par\nopagebreak
	\noindent\ignorespaces
}{
	\par\addvspace{\topsep}
}

\newcommand{\propref}[1]{\hyperref[#1]{(\textbf{#1})}}

\makeatletter
\def\@tocline#1#2#3#4#5#6#7{\relax
	\ifnum #1>\c@tocdepth 
	\else
	\par \addpenalty\@secpenalty\addvspace{#2}%
	\begingroup \hyphenpenalty\@M
	\@ifempty{#4}{%
		\@tempdima\csname r@tocindent\number#1\endcsname\relax
	}{%
		\@tempdima#4\relax
	}%
	\parindent\z@ \leftskip#3\relax \advance\leftskip\@tempdima\relax
	\rightskip\@pnumwidth plus4em \parfillskip-\@pnumwidth
	#5\leavevmode\hskip-\@tempdima
	\ifcase #1
	\or\or \hskip 1em \or \hskip 2em \else \hskip 3em \fi%
	#6\nobreak\relax
	\hfill\hbox to\@pnumwidth{\@tocpagenum{#7}}\par
	\nobreak
	\endgroup
	\fi}
\makeatother

\begin{document}
	
	\title[]{Adequate complete intersection homomorphisms}%
	\author{Samuel Alvite and Javier Majadas}%
	\address{Departamento de Matem\'aticas, Facultad de Matem\'aticas, Universidad de Santiago de Compostela, E15782 Santiago de Compostela, Spain}%
	\email{samuel.alvite.pazo@usc.es, j.majadas@usc.es}%
	
	\thanks{$^{(\star)}$ This work was partially supported by Agencia Estatal de Investigaci\'on (Spain), grant PID2024-155502NB-I00 (European FEDER support included, UE) and by Xunta de Galicia through the Competitive Reference Groups (GRC) ED431C 2023/31. Samuel Alvite was also financially supported by Xunta de Galicia Scholarship ED481A-2023-032.}


	\begin{abstract}
		We study three classes of local homomorphisms and their behavior with respect to the ascent and descent of the \emph{complete intersection} property. Crucially, they fall in between the already studied classes of complete intersection and quasi-complete intersection homomorphisms, while also repairing some of the issues these presented.
	\end{abstract}

	\maketitle
	\setcounter{secnumdepth}{1}
	\setcounter{section}{-1}
	\tableofcontents

	\section{Introduction}
	In this paper, we will consider the following problem: finding a \emph{naturally defined} ``good'' class of local homomorphisms of (commutative) noetherian local rings under which the property of being a \emph{complete intersection} ring ascends and descends. Here, it is to be understood that ``good'' means, at least, that the class is big enough (so that (a)+(b) $\Rightarrow$ (c) as in (i) below) and that other usual properties, weaker than complete intersection, also ascend and descend along the homomorphisms in this class (in the sense of (ii) and (iii) below). To make this precise, we define a class $ \mathcal{C} $ of local homomorphism of noetherian local rings to be \emph{ci-adequate} if for any $ f \colon A \to B $ in $\mathcal{C}$, the following statements hold:
	\begin{enumerate}
		\item[(i)] Any two of the following properties imply the remaining one:
		\begin{enumerate}
			\item[(a)] $ A $ is a complete intersection ring.
			\item[(b)] $ B $ is a complete intersection ring.
			\item[(c)] $ f \in \mathcal{C} $.
		\end{enumerate}
		\item[(ii)] If $ f \in \mathcal{C} $, $ A $ is Gorenstein if and only if $ B $ is Gorenstein.
		\item[(iii)] If $ f \in \mathcal{C} $, $ A $ is Cohen-Macaulay if and only if $ B $ is Cohen-Macaulay.
	\end{enumerate}
	
	An attempt at finding a \emph{ci-adequate} class could be the class of homomorphisms satisfying $ \HH_n( A, B, -)=0 $  for all $ n \geq 2 $ studied in \cite{AvramovAnnals}, where they are called \emph{complete intersection} (\emph{ci}) homomorphisms. These homomorphisms satisfy (i), (ii) and (iii) except for the fact that $ A $ and $ B $ being complete intersection does not imply that $ \HH_2( A, B, -)=0 $. It is, therefore, too restrictive of a class to be \emph{ci-adequate}.

	A second attempt, more suited to our purposes, would be the homomorphisms $ A \to B $ satisfying $ \HH_n( A, B, -)=0 $ for all $ n \geq 3 $. These homomorphisms clearly satisfy (i), as shown by the Jacobi-Zariski exact sequence and \cite[17.13]{An1974}. They also satisfy (ii) and one of the implications in (iii) ($ A $ Cohen-Macaulay $\Rightarrow$ $ B $ Cohen-Macaulay). These two claims were both proven in \cite[Corollary 6]{GarciaSoto} in the case $ A \to B $ is surjective (using the study of these homomorphisms carried out, in terms of Koszul homology, in \cite{RodicioHelv}, \cite{BlancoMajadasRodicioInventiones} and \cite{BlancoMajadasRodicioJPAA}) and were later extended to the general case in \cite{AvramovHenriquesSega}, where these homomorphisms are called \emph{quasi-complete intersection} (\emph{qci}) homomorphisms. Nevertheless, it is not known if this class is \emph{ci-adequate} or not: whether the other implication in (iii) holds is an open problem (it can be shown to be equivalent to \cite[Conjecture 4.2]{AvramovHenriquesSega}). Hence, this class of homomorphisms is, a priori, too big to be \emph{ci-adequate}.
	
	The aim of this chapter is to search for a \emph{ci-adequate} class of homomorphisms. As the previous attempts suggest, the desired class should fall in between \emph{ci} and \emph{qci} homomorphisms, while satisfying (i), (ii) and (iii). Due to the simplicity of the definitions of the aforementioned classes of homomorphisms, given by vanishing of André-Quillen homology, which has really good properties, the price to pay is clear: the sought after class will be harder to handle. We will study three classes of homomorphisms, which will be referred to as \emph{fci}, \emph{mci} and \emph{rci} and will be treated, respectively, in Sections $\S$\ref{sectionFCI}, $\S$\ref{sectionMCI} and $\S$\ref{sectionRCI}. They are also defined in terms of André-Quillen homology, they satisfy (i), (ii) and (iii) (except for the fact that \emph{fci} homomorphisms we require that $ B $ is an essentially of finite type algebra over $ A $ for the implication $ B $ Cohen-Macaulay $ \Rightarrow $ $ A $ Cohen-Macaulay), and they sit in between the classes of \emph{ci} and \emph{qci} homomorphisms, related in a way the diagram below (pertaining to Proposition \ref{relationRCIMCIFCI}) makes precise:
	
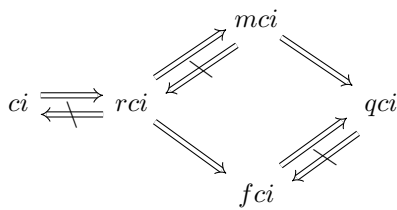
\begin{center}
		\begin{tikzcd}
			& & mci & 
			\\ci & rci & & qci
			\\ & & fci &  
			
			\arrow[from=2-1, to=2-2, Rightarrow, shift left=1.5]
			\arrow[from=2-2, to=2-1, negated, Rightarrow, shift left=1.5]
			\arrow[from=2-2, to=1-3, Rightarrow, shift left=1.5]
			\arrow[from=1-3, to=2-2, negated, Rightarrow, shift left=1.5]
			\arrow[from=2-2, to=3-3, Rightarrow]
			\arrow[from=1-3, to=2-4, Rightarrow]
			\arrow[from=3-3, to=2-4, Rightarrow, shift left=1.5]
			\arrow[from=2-4, to=3-3, negated, Rightarrow, shift left=1.5]
		\end{tikzcd}
	\end{center}

	\section{Preliminaries} \label{preliminaries}
	All rings and algebras will be commutative. We briefly recall the definition of André-Quillen homology and, for the sake of clarity in the exposition of the following sections and ease of reference, also include some properties satisfied by the  André-Quillen homology modules. Denote by $\mathbb{L}_{B|A}$ the cotangent complex of an $A$-algebra $B$, and if $M$ is a $B$-module, by  $\HH_i(A,B,M)=\HH_i(\mathbb{L}_{B|A}\otimes_B M)$ the Andr\'e-Quillen homology modules \cite{An1974}, \cite{Quillen}. We state some of properties of these homology modules, as they will be used repeatedly:\\
	
	\begin{property}{Degree 0 homology}[\cite[6.3]{An1974}]{0DG}
		$\HH_0(A,B,M)=\Omega_{B|A}\otimes_B M$. In particular, $ \HH_0(A,B,M)=0 $ for $ A \to B $ surjective.\\
	\end{property} 
	
	\begin{property}{Degree 1 homology}[\cite[6.1]{An1974}]{1DG}
		If $B=A/I$, then $\HH_1(A,B,M)=I/I^2\otimes_B M$.\\
	\end{property}
	
	\begin{property}{Localization}[\cite[4.59, 5.27]{An1974}]{L}
		Let $f:A \rightarrow B$ be a ring homomorphism, $T$ a multiplicative subset of $B$, $S$ a multiplicative subset of $A$ such that $f(S)\subset T$, and $M$ a $B$-module. Then
		\[ T^{-1}\HH_n(A,B,M) = \HH_n(A,B,T^{-1}M)= \]
		\[ \HH_n(A,T^{-1}B,T^{-1}M) = \HH_n(S^{-1}A,T^{-1}B,T^{-1}M)  \]
	\end{property}
	
	\begin{property}{Base change}[\cite[4.54]{An1974}]{BC}
		Let  $A \to B$, $A \to C$ be ring homomorphisms such that $\Tor_i^A(B,C)=0$ for all $i > 0$, and let $M$ be a $B\otimes_AC$-module. Then
		\[\HH_n(A,B,M)=\HH_n(C,B\otimes_AC,M)\]
		for all $n$. \\
	\end{property}
	
	\begin{property}{Change of coefficients}[\cite[3.20]{An1974}]{CC}
		Let $B$ be an $A$-algebra, $C$ a $B$-algebra, $M$ a flat $C$-module. Then from the definition we obtain
		\[ \HH_n(A,B,M)=\HH_n(A,B,C) \otimes_C M \]
		for all $n$.\\
	\end{property}
	
	\begin{property}{Jacobi-Zariski exact sequence}[\cite[5.1]{An1974}]{JZ}
		If $A \rightarrow B \rightarrow C$ are ring homomorphisms and $M$ is a $C$-module, we have a natural exact sequence 
		
		\begin{align*}
			&\cdots \rightarrow \HH_n(A,B,M) \rightarrow \HH_n(A,C,M) \rightarrow \HH_n(B,C,M) \rightarrow
			\HH_{n-1}(A,B,M) \rightarrow \cdots 
			\\&\cdots \rightarrow \HH_0(A,B,M) \rightarrow \HH_0(A,C,M) \rightarrow \HH_0(B,C,M) \rightarrow 0
		\end{align*}
	\end{property}
	
	\begin{property}{Regular sequences}[\cite[6.25]{An1974}]{RS}
		Let $ \locala $ be a noetherian local ring and $ I $ an ideal of $ A $. The following are equivalent:
		\begin{enumerate}
			\item[(i)] $ I $ is generated by a regular sequence.
			\item[(ii)] $ \HH_n( A, A/I, M)=0 $ for all $ n \geq 2 $ and any $ A/I $-module $ M $.
			\item[(iii)] $ \HH_2( A, A/I, k)=0 $. \\
		\end{enumerate}
	\end{property}
	
	\begin{property}{Regularity}[\cite[6.26]{An1974}]{Reg}
		As an immediate consequence of \propref{RS}, we obtain the following characterization of regularity.\\
		Let $ \locala $ be a noetherian local ring. The following are equivalent:
		\begin{enumerate}
			\item[(i)] $ A $ is a regular local ring.
			\item[(ii)] $ \HH_n( A, k, M)=0 $ for all $ n \geq 2 $ and any $ k $-module $ M $.
			\item[(iii)] $ \HH_2( A, k, k)=0 $. \\
		\end{enumerate}
	\end{property}
	
	\begin{property}{Complete intersection}[\cite[6.27, 6.28]{An1974}]{CI}
		Again as a consequence of \propref{RS}, \propref{JZ} and \propref{C} below, but also by means of \cite[17.13]{An1974}, we obtain the following characterization of complete intersections.\\
		Let $ \locala $ be a noetherian local ring. The following are equivalent:
		\begin{enumerate}
			\item[(i)] $ A $ is complete intersection.
			\item[(ii)] $ \HH_n( A, k, M)=0 $ for all $ n \geq 3 $ and any $ k $-module $ M $.
			\item[(iii)] $ \HH_3( A, k, k)=0 $.
			\item[(iv)] $ \HH_4( A, k, k)=0 $. \\
		\end{enumerate}
	\end{property}

	\begin{property}{Field extensions}[\cite [7.4]{An1974}]{FEx}
		If $K\rightarrow L$ is a field extension and $M$ an $L$-module, we have $\HH_n(K,L,M)=0$ for all $n \geq 2$. So if $A\rightarrow K \rightarrow L$ are ring homomorphisms with $K$ and $L$ fields, from \propref{JZ} we obtain $\HH_n(A,K,L) = \HH_n(A,L,L)$ for all $n \geq 2$, which, using \propref{CC}, gives $\HH_n(A,K,K) \otimes_KL=\HH_n(A,L,L)$ for all $n \geq 2$.\\
	\end{property}
	
	\begin{property}{Completion}[\cite[10.18]{An1974}]{C}
		Let $ \locala $ be a noetherian local ring, $ (\widehat{A}, \widehat{\mathfrak{m}},k )$ its $ \mathfrak{m} $-adic completion and $ M $ a $ k $-module. Then there exist natural isomorphisms
		\[ \HH_n( A, k, M) = \HH_n( \widehat{A}, k, M) \]
		for all $ n \geq 0 $. As a consequence, \propref{JZ} yields immediately that
		\[ \HH_n( A, \widehat{A}, M) = 0 \]
		for all $ n \geq 0 $ and any $ k $-module $ M $. Moreover, for any local homomorphism $ \locala \to \localb $ of noetherian local rings and an $ l $-module $ M $, again using \propref{JZ} we obtain
		\[ \HH_n( A, B, M) = \HH_n( A, \widehat{B}, M) = \HH_n( \widehat{A}, \widehat{B}, M) \]
		for all $ n \geq 0 $.\\
	\end{property}
	
	Additionally, throughout this paper we will make frequent use of \emph{regular} and \emph{Cohen factorizations} of a local homomorphism of noetherian local rings. These were introduced by Avramov, Foxby and Herzog in \cite{AvramovFoxbyHerzog}. Let $ \locala \to \localb $ be a local homomorphism of noetherian local rings:
	\begin{enumerate}
		\item[(i)] A \emph{regular factorization} of $ f $ is a diagram $ A \xrightarrow{\tau} R \xrightarrow{\sigma} B $ such that $ \sigma \tau = f $, where $ \tau $ is a flat local homomorphism of noetherian local rings with regular closed fibre $ R \otimes_A k $, and $ \sigma $ is surjective.
		\item[(ii)] A \emph{Cohen factorization} of $ f $ is regular factorization $ A \xrightarrow{\tau} R \xrightarrow{\sigma} B $ of $ f $ where, moreover, $ R $ is a complete local ring.
	\end{enumerate}
	
	Cohen factorizations exist, as was shown in \cite[Theorem 1.1]{AvramovFoxbyHerzog}, whenever the target local ring is complete. That is, for any local homomorphism $f \colon \locala \to \localb $, there is a Cohen factorization of $ f' \colon A \to \widehat{B} $,
	\begin{center}
		\begin{tikzcd}
			& R & 
			\\A & & \widehat{B}
			
			\arrow[from=2-1, to=2-3, "f'"]
			\arrow[from=2-1, to=1-2, "\tau"]
			\arrow[from=1-2, to=2-3, "\sigma"]
		\end{tikzcd}
	\end{center}
	
	\section{Feebly complete intersection homomorphisms} \label{sectionFCI}
	
	\begin{defn}\label{defMCI}
		A local homomorphism $ f \colon \locala \to \localb $ of noetherian local rings is said to be \emph{feebly complete intersection} (\emph{fci} for short) if there exists a flat local homomorphism of noetherian local rings $ A \to A' $ and a local homomorphism of noetherian local rings $ Q \to A' $ such that there exists a maximal ideal $ \mathfrak{n}' $ of $ B \otimes_A A' $ contracting to the maximal ideals of $ A' $ and $ B $ satisfying
		\[ \HH_n(Q,A', l')=0=\HH_n(Q, B \otimes_A A', l') \]
		for all $ n \geq 2 $, where $ l'=(B \otimes_A A')/\mathfrak{n}' $.
	\end{defn}
	
	Note that there always exists a prime ideal of $ B \otimes_A A' $ contracting to the maximal ideals of $ A' $ and $ B $ (see for instance \cite[3.2.7.1 (i)]{EGAISpri}), and so too does such a maximal ideal exist.
	
	\begin{prop}\label{strictImplications}
		Let $ f \colon A \to B $  be a local homomorphism of noetherian local rings. There are strict implications
		\[ \HH_n(A,B, l)=0 \quad \forall \, n \geq 2 \Rightarrow f \text{ is \emph{fci}} \Rightarrow \HH_n(A,B, l)=0 \quad \forall \, n \geq 3 \]
		where $ l $ is the residue field of $ B $.
	\end{prop}
	\begin{proof}
		First, if $ \HH_n(A,B, l)=0 $ for $ n \geq 2 $, taking $ A' = A, Q=A $ one sees immediately that $ f $ is an \emph{fci} homomorphism. The converse is, however, not true, since taking $ A $ to be a complete intersection but not regular ring and $ B $ its residue field yields $ \HH_2(A,B,l) \neq 0 $, but taking the completion $ A'=\widehat{A} $ and $ Q $ a regular local ring such that $ \widehat{A} = Q/I $, one sees that $ f $ is \emph{fci} using \propref{RS} and \propref{Reg}: for all $ n \geq 2 $ 
		\[ \HH_n(Q, \widehat{A}, l)=0, \qquad  \HH_n(Q, B \otimes_A \widehat{A}, l)=\HH_n(Q, l, l)=0.\]
		
		Now assume that $ f $ is \emph{fci}. Let $ A \to A' $, $ Q \to A' $ and $ l' $ be as in Definition \ref{defMCI}. The Jacobi-Zariski exact sequence associated with $ Q \to A' \to B \otimes_A A' $ gives
		\[ \HH_n(A', B \otimes_A A', l')=0 \]
		for all $ n \geq 3 $. By flat base change,
		\[ \HH_n(A, B, l')=0 \quad \forall n  \geq 3\]
		and so, by \propref{CC} one obtains
		\[ \HH_n(A, B, l)=0 \quad \forall n  \geq 3. \]
		
		We defer to Remark \ref{counterExampleH3fci} the proof that the converse to this does not hold.
	\end{proof}
	
	\begin{rem}
		Let $ f \colon \locala \to \localb $ be a local homomorphism of noetherian local rings. If $ f $ is \emph{fci}, then the homomorphism
		\[ \gamma_n \colon \HH_n(A, k, l) \to \HH_n(B, l, l) \]
		is an isomorphism for all $ n \geq 3 $. Indeed, the vanishing of $ \HH_n(A,B,l) $ for $ n\geq 3 $ is tantamount to $ \gamma_n $ being an isomorphism for $ n \geq 3 $: first, recall that there is an isomorphism $ \HH_n( A, k,l ) = \HH_n( A, l, l) $ given by \propref{FEx}. Now, necessity follows directly from the Jacobi-Zariski exact sequence associated with $ A \to B \to l $, while sufficiency requires a little more care. Namely, Jacobi-Zariski allows one to conclude that $ \gamma_n $ is an isomorphism for $ n \geq 4 $, but for $ n=3 $ consider the exact sequence
		\[ 0=\HH_3( A, B, l) \to \HH_3( A, l, l) \xrightarrow{\gamma_3} \HH_3( B, l, l) \xrightarrow{\alpha} \HH_2( A, B, l) \to \HH_2( A, l, l), \]
		in which we need only show $\alpha=0$. Take a regular factorization $ A \to R \to \widehat{B} $ and observe that by \cite[Lemma 19]{MajadasNagoya} one has $ \HH_n( A, l, l) = \HH_n( R, l, l) $ and $ \HH_n( A, B, l) = \HH_n( R, \widehat{B}, l) $ for $ n \geq 2 $. Thus the above exact sequence is isomorphic to the Jacobi-Zariski exact sequence associated with $ R \to \widehat{B} \to l $, so the desired result is just \cite[Corollary 6]{RodicioJPAA}.

		However, it is possible to give a more elementary proof of the fact that $ f $ being \emph{fci} $\Rightarrow$ $ \gamma_n $ is an isomorphism for $ n \geq 3 $ without needing to use \cite[Corollary 6]{RodicioJPAA}. For $ n \geq 3 $ we have a commutative diagram with exact rows 
		\begin{center}
			\begin{tikzcd}[column sep=small, scale cd=0.85]
				0=\HH_n(Q,A',l') \arrow[r] & \HH_n(Q,l',l') \arrow[r] & \HH_n(A',l',l') \arrow[r] & \HH_{n-1}(Q,A',l')=0
				\\ 0=\HH_n(Q,B \otimes_A A',l') \arrow[r] & \HH_n(Q,l',l') \arrow[r] & \HH_n(B \otimes_A A',l',l') \arrow[r] & \HH_{n-1}(Q,B \otimes_A A',l')=0
				
				\arrow[from=1-2, to=2-2, equal]
				\arrow[from=1-3, to=2-3]
			\end{tikzcd}
		\end{center}
		by naturality of the Jacobi-Zariski exact sequence. This shows that the homomorphism
		\[ \alpha_n \colon \HH_n(A',l',l') \to \HH_n(B \otimes_A A', l',l') \]
		is an isomorphism for $ n \geq 3 $.
		
		Similarly, we can obtain a commutative diagram with exact rows 
		\begin{center}
			\begin{tikzcd}[column sep=small, scale cd=0.72]
				\HH_{n+1}(A',l',l') \arrow[r] & \HH_n(A,A',l') \arrow[r] & \HH_n(A,l',l') \arrow[r] & \HH_n(A',l',l') \arrow[r] & \HH_{n-1}(A,A',l')
				\\ \HH_{n+1}(B \otimes_A A', l', l') \arrow[r] & \HH_n(B ,B \otimes_A A',l') \arrow[r] & \HH_n(B,l',l') \arrow[r] & \HH_n(B \otimes_A A', l',l') \arrow[r] & \HH_{n-1}(B,B \otimes_A A',l')
				
				\arrow[from=1-1, to=2-1, "\alpha_{n+1}"]
				\arrow[from=1-2, to=2-2, "\beta_n"]
				\arrow[from=1-3, to=2-3, "\gamma_n'"]
				\arrow[from=1-4, to=2-4, "\alpha_n"]
				\arrow[from=1-5, to=2-5, "\beta_{n-1}"]
			\end{tikzcd}
		\end{center}
		where $ \alpha_n $ is an isomorphism for all $ n \geq 3 $ and by flat base change $ \beta_n $ is an isomorphism for all $ n $. Hence, $ \gamma_n' $ is an isomorphism for $ n \geq 3 $ and, since
		\begin{align*}
			\HH_n(A,k,l) \otimes_l l' &= \HH_n(A,l',l)
			\\ \HH_n(B,l,l) \otimes_l l' &= \HH_n(B,l',l')
		\end{align*}
		for $ n \geq 2 $, we can conclude.
	\end{rem}
	
	\begin{lem}\label{vanishingRegularCI}
		Let $ f \colon A \to B $ be a local homomorphism of noetherian local rings.
		\begin{enumerate}
			\item[(i)] If $ A $ is regular and $ B $ is complete intersection, then $ \HH_n(A,B, -)=0 $ for all $ n \geq 2 $.
			\item[(ii)] If $ A $ and $ B $ are complete intersections, then $ \HH_n(A,B, -)=0 $ for all $ n \geq 3 $.
		\end{enumerate}
	\end{lem}
	\begin{proof}
		(i). By \cite[Supplément, Proposition 29]{An1974}, it suffices to show that for all prime ideals $ \mathfrak{q} $ of $  B $,
		\[ \HH_n(A,B, k(\mathfrak{q}))=0 \]
		for all $ n \geq 2 $, where $ k(\mathfrak{q}) $ is the residue field of $ B_{\mathfrak{q}} $. Taking $ \mathfrak{p} \eqdef f^{-1}(\mathfrak{q}) $ one has $ \HH_n(A, B, k(\mathfrak{q})) = \HH_n(A_{\mathfrak{p}}, B_{\mathfrak{q}}, k(\mathfrak{q})) $ by \propref{L}. Now, Serre's theorem gives that $ A_{\mathfrak{p}} $ is a regular local ring, while \cite{AvramovCI} shows that $ B_\mathfrak{q} $ is complete intersection. In the Jacobi-Zariski exact sequence
		\[ \HH_3(B_\mathfrak{q}, k(\mathfrak{q}), k(\mathfrak{q})) \to \HH_2(A_\mathfrak{p}, B_\mathfrak{q}, k(\mathfrak{q})) \to \HH_2(A_\mathfrak{p}, k(\mathfrak{q}), k(\mathfrak{q}))\]
		the leftmost term vanishes, $ B_{\mathfrak{q}} $ being complete intersection \propref{CI}. The rightmost term, $ \HH_2(A_\mathfrak{p}, k(\mathfrak{q}), k(\mathfrak{q})) = \HH_2(A_\mathfrak{p}, k(\mathfrak{p}), k(\mathfrak{p})) \otimes_{k(\mathfrak{p})} k(\mathfrak{q})$, also vanishes, since $ A_{\mathfrak{p}} $ is regular \propref{Reg}. Thus
		\[ \HH_n(A_\mathfrak{p}, B_{\mathfrak{q}}, k(\mathfrak{q}))=0 \]
		for all $ n \geq 2 $ as desired.
		
		The proof of (ii) is similar to that of (i), now using the exact sequence
		\[  \HH_4(B_\mathfrak{q}, k(\mathfrak{q}), k(\mathfrak{q})) \to \HH_3(A_\mathfrak{p}, B_\mathfrak{q}, k(\mathfrak{q})) \to \HH_3(A_\mathfrak{p}, k(\mathfrak{q}), k(\mathfrak{q})). \]
	\end{proof}
	
	\begin{thm} \label{2of3}
		Let $ f \colon \locala \to \localb $ be a local homomorphism of noetherian local rings. Two out of the following conditions imply the third:
		\begin{enumerate}
			\item[(i)] $ f $ is \emph{fci}.
			\item[(ii)] $ A $ is a complete intersection ring.
			\item[(ii)] $ B $ is a complete intersection ring.
		\end{enumerate}
		
		In particular, a noetherian local ring $ \locala $ is complete intersection if and only if the homomorphism $ A \to k $ is \emph{fci}.
	\end{thm}
	\begin{proof}
		If $ f $ is \emph{fci}, by Proposition \ref{strictImplications} $ \HH_n(A,B,l)=0 $ for all $ n \geq 3 $, so the exact sequence
		\[  \HH_4(A, B, l) \to \HH_4(A, l, l) \to \HH_4(B, l, l) \to \HH_3(A, B, l)  \]
		shows that $ A $ is complete intersection if and only if $ B $ is, since a local ring $ (R,l) $ is complete intersection if and only if $ \HH_4(R, l, l)=0 $ by \propref{CI}.
		
		Assume now that both $ A $ and $ B $ are complete intersection. Let $  Q  $ be a regular local ring such that $ \widehat{A} = Q/I $. Lemma \ref{vanishingRegularCI} (or \propref{RS}) yields
		\[ \HH_n(Q, \widehat{A}, -)=0 \]
		for all $ n \geq 2 $. Let $ \mathfrak{n}' $ be a maximal ideal of $ B \otimes_A \widehat{A} $ contracting to the maximal ideals of $ \widehat{A} $ and $ B $, and set $ l'=(B \otimes_A \widehat{A})/\mathfrak{n}' $. We have to show that
		\[ \HH_n(Q, B \otimes_A \widehat{A}, l')=0 \]
		for all $ n \geq 2 $.
		
		By \cite[5.21]{An1974} we have an exact sequence
		\begin{align*}
			\HH_n( B, l', l') &\oplus \HH_n( \widehat{A}, l', l') \to \HH_n( B \otimes_A \widehat{A}, l', l') 
			\\ &\to \HH_{n-1}( A, l', l') \to \HH_{n-1}( B, l', l') \oplus \HH_{n-1}( \widehat{A}, l', l')
		\end{align*}
		in which the leftmost term vanishes for all $ n \geq 3 $ since $ B $ and $ \widehat{A} $ are complete intersection \propref{CI}. Additionally, the homomorphism
		\[ \HH_{n-1}( A, l', l') \to \HH_{n-1}( \widehat{A}, l', l') \]
		is injective (it is, in fact, an isomorphism \propref{C}). Therefore, we obtain
		\[ \HH_n( B \otimes_A \widehat{A}, l', l')=0 \quad \forall n \geq 3. \]
		Now, regularity of $ Q $ means $ \HH_n( Q, l', l')=0 $ for $ n \geq 2 $, so the Jacobi-Zariski exact sequence associated with $ Q \to B \otimes_A \widehat{A} \to l' $ gives
		\[ \HH_n( Q, B \otimes_A \widehat{A}, l')=0 \]
		for all $ n \geq 2 $.
	\end{proof}
	
	\begin{lem} \label{LemmaEquivDef}
		Let $ f \colon \locala \to \localb $ be a local homomorphism of noetherian local rings. The following are equivalent:
		\begin{enumerate}
			\item[(i)] $ f $ is \emph{fci}
			\item[(ii)] There exists a flat and local homomorphism of noetherian local rings $ A \to A' $, a \emph{surjective} homomorphism of noetherian local rings $ Q \to A' $ and a maximal ideal $ \mathfrak{n}' $ of  $ B \otimes_A A' $ contracting to the maximal ideals of $ A' $ and $ B $ satisfying, for all $ n \geq 2 $,
			\[ \HH_n( Q, A', l')=0=\HH_n( Q, B \otimes_A A', l') \]
			where $ l' = (B \otimes_A A')/\mathfrak{n}' $
			\item[(iii)] There exists a flat and local homomorphism of noetherian local rings $ A \to A' $, a local homomorphism of noetherian local rings $ Q \to A' $ and a maximal ideal $ \mathfrak{n}' $ of  $ B \otimes_A A' $ contracting to the maximal ideals of $ A' $ and $ B $ satisfying, for $ l' = (B \otimes_A A')/\mathfrak{n}' $,
			\begin{align*}
				\HH_n( Q, A', l')=0 \quad \text{for all } n \geq 3
				\\ \HH_n( Q, B \otimes_A A', l')=0 \quad \text{for all } n \geq 2
			\end{align*}
		\end{enumerate}
	\end{lem}
	\begin{proof}
		(i) $ \Rightarrow $ (ii). Let $ A \to A' $, $ Q \to A' $ and $ l' $ be as in Definition \ref*{defMCI}. Let $ Q \to Q' \to \widehat{A'} $ be a regular factorization, that is, $ Q \to Q' $ is a flat local homomorphism with regular closed fibre and $ Q' \to \widehat{A'} $ is surjective. 
		By \cite[Lemma 1.7]{AvramovAnnals}, we have
		\[ \HH_n( Q', \widehat{A'}, l'') = \HH_n( Q, A', l'')=0 \quad \text{for all } n \geq 2,\]
		where $ l'' $ is the residue field at a maximal ideal of $ B \otimes_A \widehat{A'} $ contracting to the ideal $ \mathfrak{n}' $ of $ B \otimes_A A' $ and the maximal ideal of $ \widehat{A'} $ (this ideal exists, for instance, by applying \cite[3.2.7.1 (i)]{EGAISpri} to the following pushout diagram
		\begin{center}
			\begin{tikzcd}
				A \arrow[r] \arrow[d] & A' \arrow[r] \arrow[d] & \widehat{A'} \arrow[d]
				\\B \arrow[r] & B \otimes_A A' \arrow[r] & B \otimes_A \widehat{A'}
			\end{tikzcd}
		\end{center}
		since the left square yields, after taking a maximal ideal, the $\mathfrak{n}'$ from Definition \ref{defMCI}, and then the right square gives a prime lying over the maximal ideal $\widehat{\mathfrak{m'}}$ of $ \widehat{A'} $ and $\mathfrak{n}'$, so it suffices again to take a maximal ideal containing it). Furthermore, since $ A' \to \widehat{A'} $ is flat, by base change \propref{BC} and \propref{C} one gets
		\[ \HH_n( B \otimes_A A', B \otimes_A \widehat{A'}, l'')=\HH_n( A', \widehat{A'}, l'') =0 \quad \forall n \geq 2 \]
		Thus, from the Jacobi-Zariski exact sequence associated with $ Q \to B \otimes_A A' \to B \otimes_A \widehat{A'} $,
		\[ \HH_n( Q, B \otimes_A \widehat{A'}, l'')=0 \quad \forall n \geq 2. \]
		
		Lastly, using \cite[Lemma 19]{MajadasNagoya},
		\[ \HH_n( Q', B \otimes_A \widehat{A'}, l'')=\HH_n( Q, B \otimes_A \widehat{A'}, l'')=0 \]
		so replacing $ A' $ by $ \widehat{A'} $ and $ Q $ by $ Q' $ we obtain the desired result.
		
		(iii) $ \Rightarrow $ (i). Let $ A \to A' $, $ Q \to A' $ and $ l' $ be as in (iii). By the same reasoning as in the proof of (i) $ \Rightarrow $ (ii), we can assume $ Q \to A' $ to be surjective. It is then enough to show that $ \HH_2( Q, A', l')=0 $.
		
		Since $ \HH_3( Q, A', l')=0 $, we have an injective homomorphism
		\[ \HH_2( Q, A', l') \to \HH_2( Q, l', l') \]
		by \cite[Corollary 6]{RodicioJPAA}. But this homomorphism factors through $ \HH_2( Q, B \otimes_A A', l')=0 $, so we get 
		\[ \HH_2( Q, A', l')=0. \]
		
		Both (ii) $ \Rightarrow $ (i) and (i) $ \Rightarrow $ (iii) are trivial.
	\end{proof}
	
	\begin{rem}
		Let $ f \colon \locala \to (B, \mathfrak{n}, k) $ be a surjective homomorphism of noetherian local rings. Let $ I = \ker(f) $ and $\HH^{Kos}_1(I) $ the first Koszul homology module associated with a set of generators of I. Note that the fact that $\HH^{Kos}_1(I) $ is free as a $ B $-module is does not depend on the choice of the set of generators of $ I $, as shown for instance in \cite[Prop. 1.6.21]{BruHer}, since that proof can be extended to non-necessarily finite sets of generators. Indeed, let $ \{x_u\}_{u \in U} $, $ \{x'_w\}_{w \in W} $ be sets of generators of $ I $,
		\begin{align*}
			K =& A \langle\{X_u\}_{u \in U} ; dX_u = x_u\rangle
			\\\bar{K} =& A \langle\{X_u\}_{u \in U} \cup \{X'_w\}_{w \in W} ; dX_u = x_u, dX'_w = x'_w\rangle
		\end{align*}
		the associated Koszul complexes, and
		\[ \widetilde{K} = A \langle\{X_u\}_{u \in U} \cup \{X'_w\}_{w \in W} ; dX_u = x_u, dX'_w = 0\rangle \cong K \otimes_A \bigwedge^{*}_A A^W. \]
		For every $ w \in W $, let $ x'_w = \sum_{u} a_{uw} x_u$. The homomorphism of $ A $-algebras
		\[ \bar{K} \to \widetilde{K}, \quad X_u \mapsto X_u, \quad X'_w \mapsto X'_w + \sum_{u} a_{uw} X_u \]
		is a homomorphism of DG $ A $-algebras with inverse given by $ X_u \mapsto X_u $, $ X'_w \mapsto X'_w - \sum_{u} a_{uw} X_u $. This shows that
		\[ \HH_{*}(\bar{K}) = \HH_{*}(K) \otimes_B \bigwedge^{*}_B B^W. \]
	\end{rem}
	
	In what follows, we will denote by CI-dim the complete intersection dimension, introduced in \cite{AvramovGasharovPeeva}.
	
	\begin{prop} \label{fciKoszul}
		Let $ f \colon \locala \to (B, \mathfrak{n}, k) $ a surjective homomorphism of noetherian local rings with $ I = \ker(f) $. The following are equivalent:
		\begin{enumerate}
			\item[(i)] $ f $ is \emph{fci}
			\item[(ii)] $ \HH^{Kos}_1(I) $ is a free $ B $-module and $ \cidim_A(B) < \infty $
		\end{enumerate}
	\end{prop}
	\begin{proof}
		Assume that $ f $ is \emph{fci}. Let $ A \to A' $, $ Q \to A' $ surjective, $ l' $ be as in Lemma \ref{LemmaEquivDef} (ii). Since  $ \HH_2( Q, A', l')=0=\HH_2( Q, B \otimes_A A', -) $, the kernels of the surjective maps $ Q \to A' $ and $ Q \to B \otimes_A A' $ are generated by regular sequences. In particular, $ \pdim_Q(B \otimes_A A')< \infty $ and thus $ \cidim_A(B) < \infty $.
		
		Moreover, by Proposition \ref{strictImplications}, $ \HH_n( A, B, k)=0 $ for all $ n \geq 3 $ and therefore $ \HH^{Kos}_1(I) $ is a free $ B $-module (cf. \cite[Proposition]{AnPairs}).
		
		Conversely, assume that $ \HH^{Kos}_1(I) $ is free as a $ B $-module and $ \cidim_A(B) < \infty $. One then has a flat local homomorphism $ A \to A' $ of noetherian local rings and a surjective homomorphism of noetherian local rings $ Q \to A' $ such that $ \HH_n( Q, A', k)=0 $ for all $  n \geq 2 $ and $ \pdim_Q(B \otimes_A A')< \infty $. By \cite[Proposition 12]{Soto}, $ \HH_n( A, B, k)=0 $ for all $ n \geq 3 $, so then by (flat) base change one gets $ \HH_n( A', B \otimes_A A', k)=0 $ for all $ n \geq 3 $.
		
		The Jacobi-Zariski exact sequence associated with $ Q \to A' \to B \otimes_A A' $ yields
		\[ \HH_n( Q, B \otimes_A A', k)=0 \]
		for all $ n \geq 3 $, and now since $ \pdim_Q(B \otimes_A A')< \infty $, 
		\[ \HH_n( Q, B \otimes_A A', k)=0 \]
		for all $ n \geq 2 $ by \cite[17.2]{An1974} (or by the much stronger result \cite[Theorem A]{BriggsIyengar}).
	\end{proof}
	
	\begin{rem} \label{counterExampleH3fci}
		We now finish the proof of Proposition \ref{strictImplications}. By \cite[Theorem 3.5]{AvramovHenriquesSega}, there exists a surjective homomorphism of noetherian local rings $ A \to B $ such that $ \HH_n( A, B, -)=0 $ for all $ n \geq 3 $, but so that $ \cidim_A(B)= \infty $. But now Proposition \ref{fciKoszul} says this homomorphism $ f $ is not \emph{fci}.
	\end{rem}
	
	\begin{thm} \label{ascentDescentFCI}
		Let $ f \colon A \to B $ be a local homomorphism of noetherian local rings. Assume that $ f $ is \emph{fci}. Then:
		\begin{enumerate}
			\item[(i)] $ A $ is Gorenstein if and only if $ B $ is.
			\item[(ii)] If $ A $ is Cohen-Macaulay, then so is $ B $.
			\item[(iii)] If $ f $ is essentially of finite type and $ B $ is Cohen-Macaulay, then so is $ A $.
		\end{enumerate}
	\end{thm}
	\begin{proof}
		In \cite[Corollary 6]{GarciaSoto}, it is shown that if a surjective homomorphism of noetherian local rings $ f \colon \locala \to (B, \mathfrak{n}, k) $ satisfies $ \HH_n( A, B, k)=0 $ for all $ n \geq 3 $, then it satisfies (i) and (ii). This result was extended to the case where $ f $ is not necessarily surjective in \cite{AvramovHenriquesSega}. Therefore, by Proposition \ref{strictImplications}, it is enough to show that (iii) holds.
		
		Let $ A \to A' $, $ Q \to A' $, $ \mathfrak{n}' $ and $ l' $ as in Definition \ref{defMCI}. Note that the ring $ B \otimes_A A' $  is noetherian. For any noetherian ring $ R $, we denote by $ \mathrm{coprof}(R) = \dim(R) - \mathrm{depth}(R) $ \cite[$0_{\rm IV}$.16.4.9]{EGAIV1}, so that $ R $ is Cohen-Macaulay if and only if $ \mathrm{coprof}(R) =0 $. We will show that
		\[ \mathrm{coprof}(A) =\mathrm{coprof}(B),  \]
		so in particular $ A $ is Cohen-Macaulay if and only if $ B $ is.
		
		If $ (R, \mathfrak{m},k) \to (S, \mathfrak{n},l) $ is a local homomorphism of noetherian local rings such that $ \HH_2( R, S, l)=0 $, then $ \mathrm{coprof}(R) =\mathrm{coprof}(S) $, for by taking a regular factorization $ R \to T \to \widehat{S} $, $ \mathrm{coprof}(R) =\mathrm{coprof}(T) $ because $ R \to T $ is flat with regular closed fibre $ T \otimes_R k$ (so $ \mathrm{coprof}(T \otimes_R k) =0 $) and moreover, since $ \ker(T \to \widehat{S}) $ is generated by a regular sequence (because $ \HH_2( T, \widehat{S}, l)=0 $; see \propref{RS}), $ \mathrm{coprof}(T) =\mathrm{coprof}(\widehat{S}) =\mathrm{coprof}(S) $. In our particular case,
		\[ \mathrm{coprof}(A') =\mathrm{coprof}(Q) =\mathrm{coprof}((B \otimes_A A')_{\mathfrak{n}'}). \]
		By \cite[Corollary 2.5.(3)]{BouchibaCondeMajadas},
		\[ \mathrm{coprof}((B \otimes_A A')_{\mathfrak{n}'}) = \mathrm{coprof}(A') + \mathrm{coprof}(B) - \mathrm{coprof}(A) \]
		whence
		\[ \mathrm{coprof}(A) = \mathrm{coprof}(B). \]
	\end{proof}
	
	\begin{lem} \label{fciCompositions}
		Let $ \locala \xrightarrow{f} \localb \xrightarrow{g}  (C, \mathfrak{p}, h) $ be local homomorphisms of noetherian local rings.
		\begin{enumerate}
			\item[(i)] If $ f $ is \emph{fci} and $ \HH_n( B, C, h)=0 $ for all $ n \geq 2 $, then $ g \circ f $ is  \emph{fci}.
			\item[(ii)] If $ g \circ f $ is \emph{fci} and $ \HH_n( B, C, h)=0 $ for all $ n \geq 3 $, then $ f $ is \emph{fci}.
		\end{enumerate}
		
		In particular, $ A \to B $ is \emph{fci} if and only if $ A \to \widehat{B} $ is \emph{fci}.
	\end{lem}
	\begin{proof}
		(i). Let $ A \to A' $, $ Q \to A' $, $ \mathfrak{n}' $ and $ l' $ be as in Definition \ref{defMCI}. We have, in particular, that $ \HH_n( Q, B \otimes_A A', l')=0 $ for $ n \geq 2 $ so it is enough to show $ \HH_n( Q, C \otimes_A A', l'')=0 $ for $ n \geq 2 $, where $ l'' $ is the residue field of $ C \otimes_A A' $ at a maximal ideal which contracts to the maximal ideal of $ C $ and to $ \mathfrak{n}' $ in $ B \otimes_A A' $. The Jacobi-Zariski exact sequence
		\begin{align*}
			0=\HH_n( Q, B \otimes_A A', l'') \to \HH_n( Q, C \otimes_A A', l'') \to \HH_n( B \otimes_A A', C \otimes_A A', l'')& 
			\\=\HH_n( B, C, l'')=0&
		\end{align*}
		for $ n \geq 2 $ gives the desired result.
		
		(ii) Let $ A \to A' $, $ Q \to A' $ such that
		\[ \HH_n( Q, A', l'')=0=\HH_n( Q, C \otimes_A A', l'') \]
		for $ n \geq 2 $ where $ l'' $ is the residue field of $ C \otimes_A A' $ at a maximal ideal contracting to the maximal ideals of $ C $ and $ A' $. We have an exact sequence for $ n \geq 3 $
		\begin{align*}
			0 = \HH_n( B, C, l'') = \HH_n( B \otimes_A A', C \otimes_A A', l'')  \to &\HH_{n-1}( Q, B \otimes_A A', l'') 
			\\ &\to \HH_{n-1}( Q, C \otimes_A A', l'') =0 
		\end{align*}
		and thus
		\[ \HH_n( Q, B \otimes_A A', l'')=0 \quad \text{for all } n \geq 2. \]
		
		The particular case follows from the fact that $ \HH_n( B, \widehat{B}, l)=0 $ for all $ n $ \propref{C}.
	\end{proof}
	
	\begin{prop} \label{KoszulFreeness}
		Let $ f \colon \locala \to \localb $ a local homomorphism of noetherian local rings. For the following conditions:
		\begin{enumerate}
			\item[(i)] There exists a regular factorization $ A \to R \to \widehat{B} $ such that $  \cidim_R(\widehat{B})<\infty $ and $ \HH^{Kos}_1(I) $ is a free $ \widehat{B} $-module, where $ I = \ker(R \to \widehat{B}) $.
			\item[(ii)] $ f $ is \emph{fci}.
		\end{enumerate}
		(i) $ \Rightarrow $ (ii) holds. If $ f $ is essentially of finite type, then (ii) $ \Rightarrow $ (i) also holds.
	\end{prop}
	\begin{proof}
		(i) $ \Rightarrow $ (ii). Since $ \cidim_R(\widehat{B}) < \infty $, there exists a local flat homomorphism $ R \to R' $ and a surjective homomorphism of noetherian local rings $ Q \to R' $ such that $ \HH_n( Q, R', -)=0 $ for all $ n \geq 2 $ and $ \pdim_Q(\widehat{B} \otimes_R R')< \infty $. Set $ I' \eqdef IR' = \ker(R' \to \widehat{B} \otimes_R R') $. Taking a set of generators for $ I' $ arising from one for $ I $, one gets $ \HH^{Kos}_1(I') = \HH^{Kos}_1(I) \otimes_{\widehat{B}} (\widehat{B} \otimes_R R')$, so that it is a free $ \widehat{B} \otimes_R R' $-module. Now, in the following commutative diagram where the left column is exact and $l'$ is the residue field of $Q, R'$ and $\widehat{B} \otimes_R R',$
		\begin{center}
			\begin{tikzcd}
				\HH_2( R', \widehat{B} \otimes_R R' , l') & \HH_2( R', l', l')
				\\ \HH_2( Q, \widehat{B} \otimes_R R' , l') & \HH_2( Q, l', l')
				\\ \HH_2( Q, R' , -)=0
				
				\arrow[from=2-1, to=1-1, "\alpha"]
				\arrow[from=1-1, to=1-2, "\beta"]
				\arrow[from=2-1, to=2-2, "\gamma"]
				\arrow[from=3-1, to=2-1]
				\arrow[from=2-2, to=1-2]
			\end{tikzcd}
		\end{center}
		one has that $ \beta $ is injective because $ \HH^{Kos}_1(I') $ is free \cite[Theorem 1]{RodicioJPAA}, as is $ \alpha $ by exactness. Since $ \pdim_Q(\widehat{B} \otimes_R R')< \infty $, by \cite{AvramovCRAS} one gets that $ \gamma=0 $. In conclusion, $ \HH_2( Q, \widehat{B} \otimes_R R' , l')=0 $ so, by \propref{RS},
		\[ \HH_n( Q, \widehat{B} \otimes_R R' , l')=0 \quad \text{ for all } n \geq 2.\]
		
		By \cite[5.23]{An1974}, we have a canonical isomorphism 
		\[ \HH_n( A, \widehat{B}, l') \oplus \HH_n( A, R', l') \xrightarrow{\cong} \HH_n( A, \widehat{B} \otimes_A R', l'), \]
		while the exact sequence \cite[5.22]{An1974}
		\[ \HH_n( A, R, l') \to \HH_n( A, \widehat{B}, l') \oplus \HH_n( A, R', l') \to \HH_n( A, \widehat{B} \otimes_R R', l') \to \HH_{n-1}( A, R, l')\]
		gives that the homomorphism 
		\[ \HH_n( A, \widehat{B}, l') \oplus \HH_n( A, R', l') \to \HH_n( A, \widehat{B} \otimes_R R', l') \]
		is an isomorphism for $ n \geq 3 $ and injective for $ n=2 $, since $ \HH_n( A, R, l')= \HH_n( k, R \otimes_A k , l')=0 $ for all $  n \geq 2 $ because $ R \otimes_A k $ is regular.
		
		Thus, the canonical homomorphism $ \HH_n( A, \widehat{B} \otimes_A R', l') \to \HH_n( A, \widehat{B} \otimes_R R', l') $ is an isomorphism for all $n \geq 3 $ and injective for $n=2$. But now the Jacobi-Zariski exact sequence 
		\begin{align*}
			&\HH_n( A, \widehat{B} \otimes_A R', l') \to \HH_n( A, \widehat{B} \otimes_R R', l') \to \HH_n( \widehat{B} \otimes_A R', \widehat{B} \otimes_R R', l')
			\\ &\to \HH_{n-1}( A, \widehat{B} \otimes_A R', l') \to \HH_{n-1}( A, \widehat{B} \otimes_R R', l') 
		\end{align*}
		says that 
		\[ \HH_n( \widehat{B} \otimes_A R', \widehat{B} \otimes_R R', l')=0 \quad \text{ for all } n \geq 3.\]
		
		We can now deduce, using $\HH_n( Q, \widehat{B} \otimes_R R' , l')=0$ for all $n  \geq 2$, that
		\[ \HH_n( Q, \widehat{B} \otimes_A R' , l')=0 \quad \text{ for all } n \geq 2.\]
		
		Taking $A'=R'$ in Definition \ref{defMCI}, one sees that the homomorphism $A \to \widehat{B}$ is \emph{fci}. By Lemma \ref{fciCompositions}, $A \to B$ is \emph{fci}.
		
		(ii) $\Rightarrow$ (i). Under the additional hypothesis, we have a regular factorization $ A \to R \to B $, with $R$ an $A$-algebra essentially of finite type. Now let $A \to A'$, $ Q \to A'$, $\mathfrak{n}'$, $l'$ be as in Definition \ref{defMCI}. The fact that $\HH_n(A,R,l')=0$ for all $n \geq 2$ implies, by flat base change, that $\HH_n(A',R \otimes_A A',l')=0$, so by the Jacobi-Zariski exact sequence
		\[ \HH_n(Q,R \otimes_A A',l')=0\]
		for $n \geq 2$. Note that $R \otimes_A A'$ is a noetherian ring. Set $R' \eqdef (R \otimes_A A')_{\mathfrak{q}}$, $\mathfrak{q}$ being the contraction of $\mathfrak{n}'$. Localizing, we get \propref{L}
		\[\HH_n(Q,R',l')=0 \quad \text{ for all } n \geq 2,\]
		so, therefore,
		\[ \HH_n(\widehat{Q},\widehat{R'},l'')=0 \quad \text{ for all } n \geq 2, \]
		where $l''$ is the residue field of both $\widehat{(B \otimes_R R')_{\mathfrak{n}'}} = \widehat{(B \otimes_A A')_{\mathfrak{n}'}}$ and $\widehat{R'}$. Since $R \to R'$ is flat, the local criterion for flatness says that so is $\widehat{R} \to \widehat{R'}$. We will now show that $ \HH_n(\widehat{Q},\widehat{B} \otimes_{\widehat{R}}\widehat{R'},l'')=0 $ for all $n \geq 2$.
		
		Consider the homomorphisms
		\[Q \to B \otimes_A A' \to B \otimes_R R' \to \widehat{B} \otimes_{\widehat{R}}\widehat{R'} = \widehat{B \otimes_R R'}\]
		Applying $\HH_n$ for $n \geq 2$ to each of the maps above, one gets the vanishing of all of them:
		\begin{enumerate}
			\item $\HH_n(Q, B \otimes_A A', l'')=0$ by Definition \ref{defMCI},
			\item $\HH_n( B \otimes_A A', B \otimes_R R', l'')=0$  because $B \otimes_R R'$ is a localization of $B \otimes_A A'$,
			\item $\HH_n(B \otimes_R R', \widehat{B \otimes_R R'},l'')=0$ by \propref{C}.
		\end{enumerate}
		Therefore, using Jacobi-Zariski, we have that
		\[ \HH_n(Q, \widehat{B} \otimes_{\widehat{R}} \widehat{R'},l'')=0 \quad \text{ for all } n \geq 2\]
		and, since $ \HH_n(Q, \widehat{Q}, l'')=0 $ for all $n$, then
		\[ \HH_n(\widehat{Q}, \widehat{B} \otimes_{\widehat{R}} \widehat{R'},l'')=0 \quad \text{ for all } n \geq 2.\]
		
		Let $ \widehat{Q} \to T \to \widehat{R'} $ be a regular factorization. By \cite[Lemma 19]{MajadasNagoya}, for $ n \geq 2 $ one has $ \HH_n(T,\widehat{R'}, l'') = \HH_n(\widehat{Q},\widehat{R'}, l'')=0 $ and $ \HH_n(T, \widehat{B} \otimes_{\widehat{R}} \widehat{R'},l'') = \HH_n(\widehat{Q}, \widehat{B} \otimes_{\widehat{R}} \widehat{R'},l'')=0 $, so then $\ker(T \to \widehat{R'})$ and $\ker(T \to \widehat{B} \otimes_{\widehat{R}} \widehat{R'})$ are generated by regular sequences. In particular, $\pdim_T(\widehat{B} \otimes_{\widehat{R}} \widehat{R'})<\infty$.
		
		We have showed that $ \cidim_{\widehat{R}}(\widehat{B})< \infty $. It remains to be shown that, if $I = \ker(\widehat{R} \to \widehat{B})$, then $ \HH^{Kos}_1(I) $ is a free $\widehat{B}$-module. Now, by Proposition \ref{strictImplications} we obtain $\HH_n(A,B,l')=0$ for all $n \geq 3$, whence $\HH_n(R,\widehat{B},l')=0$ and then $\HH_n(\widehat{R},\widehat{B},l) \otimes_l l'=0$ for all $n \geq 3$. By \cite[Proposition]{AnPairs}, $ \HH^{Kos}_1(I) $ is a free $\widehat{B}$-module.
	\end{proof}
	
	\begin{rem} \label{remarkExamplesCounterexamples}
		\hfill
		\begin{enumerate}
			\item[(i)] Note that condition (i) in Proposition \ref{KoszulFreeness} does not imply that $\HH(A,B,l)=0$ for all $n \geq 2$. It is enough to take $A$ a local ring that is complete intersection but not regular and $B$ its residue field.
			\item[(ii)] If $f \colon \locala \to \localb$ is a homomorphism between complete intersection rings, then condition (i) in Proposition \ref{KoszulFreeness} is satisfied. Thus, by Proposition \ref{KoszulFreeness} and Theorem \ref{2of3}, two of the following conditions imply the third:
			\begin{enumerate}
				\item[(a)] Condition (i) in Proposition \ref{KoszulFreeness}.
				\item[(b)] $A$ is complete intersection.
				\item[(c)] $B$ is complete intersection.
			\end{enumerate} 
			\item[(iii)] Condition (i) in Proposition \ref{KoszulFreeness} is equivalent to the following: $\HH_n(A,B,l)=0$ for all $n \geq 3$ and there exists a Cohen factorization $ A \to R \to \widehat{B}$ such that $\cidim_R(\widehat{B})< \infty$. Indeed, by Proposition \ref{KoszulFreeness} and Proposition \ref{strictImplications}, it is enough to show that, if $\HH_3(A,B,l)=0$ and $ I = \ker(R \to \widehat{B}) $, then $ \HH^{Kos}_1(I) $ is a free $\widehat{B}$-module, but this follows from the fact that $ \HH_3( R, \widehat{B}, l)= \HH_3( A, \widehat{B}, l)= \HH_3( A, B, l)=0 $ and \cite[Proposition]{AnPairs}.
		\end{enumerate}
	\end{rem}
	
	The advantage that condition (i) in Proposition \ref{KoszulFreeness} has over the concept of \emph{fci} homomorphisms is that it is possible to prove a result analogous to that of Theorem \ref{ascentDescentFCI}, but dropping the finiteness hypotheses.
	
	\begin{thm}
		Let $ f \colon \locala \to \localb $ be a homomorphism satisfying the condition (i) in Proposition \ref{KoszulFreeness}. Then:
		\begin{enumerate}
			\item[(i)] $A$ is Gorenstein if and only if $ B $ is.
			\item[(ii)] $A$ is Cohen-Macaulay if and only if $ B $ is.
		\end{enumerate}
	\end{thm}
	\begin{proof}
		By Theorem \ref{ascentDescentFCI} and Proposition \ref{KoszulFreeness}, it is enough to show that if $B$ is Cohen-Macaulay, then so is $A$.
		
		Let $ A \to R \to \widehat{B} $ be a Cohen factorization. One can prove, more generally, that $ \mathrm{coprof}(A) \leq \mathrm{coprof}(B) $, assuming that $ \cidim_R(\widehat{B}) < \infty $. This is known: let $ R \to R' $ be a flat local homomorphism of noetherian local rings and $ Q \to R' $ a surjective homomorphism of noetherian local rings such that $ \HH_n( Q, R', -)=0 $ for all $ n \geq 2 $ and $ \pdim_Q(\widehat{B} \otimes_R R')< \infty $. Let $ l' $ be the residue field of $ R' $. The closed fibre $ l \otimes_R R' $ of the homomorphism $ R \xrightarrow{\beta} R' $ agrees with that of the induced homomorphism $ \widehat{B} \xrightarrow{\alpha} \widehat{B} \otimes_R R' $. We have
		\begin{align*}
			&\mathrm{coprof}(B) = \mathrm{coprof}(\widehat{B}) \overset{(1)}{=} \mathrm{coprof}(\widehat{B} \otimes_R R') - \mathrm{coprof}(l' \otimes_R R') 
			\\ &\overset{(2)}{=} \mathrm{coprof}(\widehat{B} \otimes_R R') - (\mathrm{coprof}(R') - \mathrm{coprof}(R)) 
			\\ &\overset{(3)}{\geq} \mathrm{coprof}(Q) - (\mathrm{coprof}(R') - \mathrm{coprof}(R))
			\\ &\overset{(4)}{=} \mathrm{coprof}(R) \geq \mathrm{coprof}(A)
		\end{align*}
		where (1) (resp. (2)) holds due to the fact that $\alpha$  (resp. $\beta$) is a flat local homomorphism, (3) follows from the fact that $ \pdim_Q(\widehat{B} \otimes_R R') < \infty $ (see for instance \cite[Theorem 2.2]{AvramovFoxbyHalperin}), and (4) from the equality $ \mathrm{coprof}(Q) = \mathrm{coprof}(R') $, owing to the fact that $\ker(Q \to R')$ is generated by a regular sequence.
	\end{proof}

	\section{Moderately complete intersection homomorphisms} \label{sectionMCI}
	
	\begin{prop} \label{eqCIDefect}
		Let $ f \colon \locala \to \localb $ be a local homomorphism of noetherian local rings. Let
		\begin{align*}
			A \to R \xrightarrow{r} \widehat{B}
			\\A \to S \xrightarrow{s} \widehat{B}
		\end{align*}
		be two regular factorizations of $f$. Then the integers
		\begin{align*}
			\mathrm{d}(r) &\eqdef \dim_l \HH_2( R, \widehat{B}, l) - \dim_l \HH_1( R, \widehat{B}, l) + \dim R - \dim B
			\\ \mathrm{d}(s) &\eqdef \dim_l \HH_2( S, \widehat{B}, l) - \dim_l \HH_1( S, \widehat{B}, l) + \dim S - \dim B
		\end{align*}
		agree.
	\end{prop}
	\begin{proof}
		Since $ \HH_n( R, \widehat{B}, l)=\HH_n( \widehat{R}, \widehat{B}, l) $, replacing $ R $ by $ \widehat{R} $ and $ S $ by $ \widehat{S} $, we can assume that $ R $ and $ S $ are complete, that is, that we are working with Cohen factorizations. By \cite[Theorem 1.2]{AvramovFoxbyHerzog} there is a commutative diagram of Cohen factorizations
		\begin{center}
			\begin{tikzcd}
				& R &
				\\A & T & \widehat{B}
				\\ & S & 
				
				\arrow[from=2-1, to=1-2]
				\arrow[from=2-1, to=2-2]
				\arrow[from=2-1, to=3-2]
				\arrow[from=2-2, to=1-2, "\alpha"]
				\arrow[from=2-2, to=3-2, "\beta", swap]
				\arrow[from=1-2, to=2-3, "r"]
				\arrow[from=2-2, to=2-3, "t"]
				\arrow[from=3-2, to=2-3, "s", swap]
			\end{tikzcd}
		\end{center}
		where $\alpha$ and $\beta$ are surjective homomorphisms of noetherian local rings whose kernels are generated by regular sequences. Then the Jacobi-Zariski exact sequence
		\begin{align*}
			0&= \HH_2( T, R, l) \to \HH_2( T, \widehat{B}, l) \to \HH_2( R, \widehat{B}, l) 
			\\ &\to \HH_1( T, R, l) \to \HH_1( T, \widehat{B}, l) \to \HH_1( R, \widehat{B}, l) \to 0 
		\end{align*}
		together with the equality
		\[ \dim_l \HH_1( T, R, l) = \dim_l J/J^2 \otimes_R l  = \dim T - \dim R,\]
		where $ J = \ker(T \to R) $, yields $ \mathrm{d}(r) = \mathrm{d}(t) $.
	\end{proof}
	
	\begin{defn}
		Following the notation of Proposition \ref{eqCIDefect}, define
		\[ \mathrm{d}(f) \eqdef \mathrm{d}(r). \] 
		
		If $ \pi \colon A \to k $ is the canonical projection of a noetherian local ring onto its residue field, $ \mathrm{d}(\pi) = \dim_l \HH_2( A, k, k) - \dim_l \HH_1( A, k, k) + \dim A $ agrees with the invariant $ \mathrm{d}(A) $ introduced in \cite{KiehlKunz}, which was baptized as the complete intersection defect of $ A $ in \cite{AvramovMathAnn}, since $ A $ is complete intersection $ \iff \mathrm{d}(A)=0$.  
	\end{defn}
	
	\begin{prop} \label{conditionsFCI}
		The following are equivalent:
		\begin{enumerate}
			\item[(i)] $ \HH_n( A, B, l)=0 $ for all $ n \geq 3 $ and $ \mathrm{d}(f)=0 $.
			\item[(ii)] $ \HH_n( A, B, l)=0 $ for all $ n \geq 3 $ and $ \mathrm{d}(A)=\mathrm{d}(B) $.
		\end{enumerate}
	\end{prop}
	\begin{proof}
		If $ A \to R \xrightarrow{r} \widehat{B}$ is a regular factorization of $ f $, then $ \mathrm{d}(B)=\mathrm{d}(\widehat{B}) $ and $ \mathrm{d}(A)=\mathrm{d}(R) $ (cf. \cite[Proposition 3.6]{AvramovMathAnn}), and also $ \HH_n( A, B, l)= \HH_n( R, \widehat{B}, l) $, so by replacing $ f $ by $  r $ we can assume that $ f $ is surjective. By \cite[Corollary 6]{RodicioJPAA} we get an exact sequence
		\begin{align*}
			0 &\to \HH_2( A, B, l) \to \HH_2( A, l, l) \to \HH_2( B, l, l) 
			\\  &\to \HH_1( A, B, l) \to \HH_1( A, l, l) \to \HH_1( B, l, l) \to 0
		\end{align*}
		from which it follows that $ \mathrm{d}(f) = \mathrm{d}(A) - \mathrm{d}(B) $.
	\end{proof}
	
	\begin{rem}
		Note that, in general, $ \mathrm{d}(f) \neq \mathrm{d}(A) - \mathrm{d}(B) $. For instance, if $ A $ is a noetherian local ring and $ I $ is an ideal of finite projective dimension which is not generated by a regular sequence, by \cite{AvramovCRAS} we have an exact sequence
		\[ 0 \to \HH_2( A, l, l) \to \HH_2( B, l, l) \to \HH_1( A, B, l) \to \HH_1( A, l, l) \to \HH_1( B, l, l) \to 0 \]
		where $ B \eqdef A/I $, showing that in this case
		\[ \mathrm{d}(f) = \mathrm{d}(A) - \mathrm{d}(B) + \dim_l \HH_2( A, B, l), \]
		but $\HH_2( A, B, l) \neq 0 $ since $ I $ is not generated by a regular sequence.
	\end{rem}
	
	\begin{defn}
		A local homomorphism $ f \colon \locala \to \localb $ of noetherian local rings is said to be \emph{moderately complete intersection} (\emph{mci} for short) if it satisfies the equivalent conditions in Proposition \ref{conditionsFCI}.
	\end{defn}
	
	\begin{prop}
		Two of the following conditions imply the third:
		\begin{enumerate}
			\item[(i)] $ f $ is \emph{mci}.
			\item[(ii)] $A$ is complete intersection.
			\item[(iii)] $B$ is complete intersection.
		\end{enumerate}
	\end{prop}
	\begin{proof}
		Assuming $ \HH_n( A, B, l)=0 $ for all $ n \geq 3 $, we have that $ A $  is complete intersection if and only if $ B $ is. Thus, it is enough to show that if $ A $ and $ B $ are complete intersection then $ f $ is \emph{mci}. We know already that $ \HH_n( A, B, l)=0 $ for all $ n \geq 3 $, and $ A $ and $ B $ being complete intersection just means that $ \mathrm{d}(A)=0=\mathrm{d}(B) $.
	\end{proof}
	
	\begin{lem} \label{lemmaGradeDepth}
		Let $ f \colon A \to B $ be a surjective homomorphism of noetherian local rings, with $ I = \ker(f) $. If $ \HH_n( A, B, l)=0 $ for all $ n \geq 3 $, then:
		\begin{enumerate}
			\item[(i)] $ \mathrm{d}(f) = \dim A - \dim B - \mathrm{grade}(I) $.
			\item[(ii)] $ \mathrm{depth}(A) - \mathrm{depth}(B) = \mathrm{grade}(I) $.
		\end{enumerate}
	\end{lem}
	\begin{proof}
		(i). Let $ \HH^{Kos}_1(I) $ be the first Koszul homology module associated with a set of generators of $ I $. By \cite[Corollary 3']{BlancoMajadasRodicioJPAA}, $ \HH^{Kos}_1(I) $ is a free $ B $-module and the canonical homomorphism
		\[ \bigwedge^{*}_B \HH^{Kos}_1(I) \to \HH^{Kos}_{\ast}(I)  \]
		is an isomorphism. Therefore, by \cite[1.6.17(b)]{BruHer}, one has
		\[ \mathrm{rank } \, \HH^{Kos}_1(I) = \mu(I) - \mathrm{grade}(I), \]
		where $ \mu(I) $ is the number of generators of a minimal generating set of $ I $. Since $ \HH^{Kos}_1(I) \otimes_B k \cong \HH_2( A, B, l) $ (cf. \cite[15.12]{An1974}) and $ \HH_1( A, B, l) = I/I^2 \otimes_B l $, by definition of $ \mathrm{d}(f) $
		\[ \mathrm{d}(f) = (\mu(I) - \mathrm{grade}(I)) - \mu(I) + \dim A - \dim B. \]
		
		(ii). This result appears in the proofs of Corollaries 5 and 6 of \cite{GarciaSoto}.
	\end{proof}
	
	\begin{thm} \label{ascentDescentMCI}
		Let $ f \colon \locala \to \localb $ be a local homomorphism of noetherian local rings. If $ f $ is \emph{mci}, then:
		\begin{enumerate}
			\item[(i)] $ A $ is Gorenstein $\iff$ $ B $ is Gorenstein
			\item[(ii)] $ A $ is Cohen-Macaulay $ \iff $ $ B $ is Cohen-Macaulay
		\end{enumerate}
	\end{thm}
	\begin{proof}
		As in the proof of Theorem \ref{ascentDescentFCI}, we need only show that if $ B $ is Cohen-Macaulay, then so is $ A $. By taking a Cohen factorization of $ f $, we may also assume that $ f $ is surjective. But then Lemma \ref{lemmaGradeDepth} gives $ \mathrm{coprof}(A) = \mathrm{coprof}(B) $.
	\end{proof}
	
	\begin{rem}
		Note that $ \HH_n( A, B, -)=0 $ for all $ n \geq 2 $ implies that $ f $ is \emph{mci}. Indeed, we can assume that $ f $ is surjective and, if $ I = \ker(f) $, then $ \dim_l \HH_1( A, B, l)= \dim_l (I/I^2 \otimes_B l) = \dim A - \dim B $ since $ I $ is generated by a regular sequence. Therefore, from the exact sequence
		\[ 0 \to \HH_2( A, l, l) \to \HH_2( B, l, l) \to \HH_1( A, B, l) \to \HH_1( A, l, l) \to \HH_1( B, l, l) \to 0 \]
		we can deduce that $ \mathrm{d}(A) = \mathrm{d}(B). $
		
		The converse is not true, as shown by a surjective homomorphism of complete intersections whose kernel is not generated by a regular sequence (e.g. $ A $ being complete intersection but not regular and $ B $ its residue field). 
	\end{rem}
	
	\begin{rem}
		We do not know whether every homomorphism $ f \colon A \to B $ such that $ \HH_n( A, B, l)=0 $ for all $ n \geq 3 $ is \emph{mci}. By Lemma \ref{lemmaGradeDepth}, this question is equivalent to Conjecture 4.2 in \cite{AvramovHenriquesSega}.
	\end{rem}
	
	\section{Reasonably complete intersection homomorphisms} \label{sectionRCI}
	In this section, we study another class of homomorphisms, which we call \emph{reasonably complete intersection} (\emph{rci} for short) homomorphisms. We will show in Proposition \ref{relationRCIMCIFCI} that every \emph{rci} homomorphism is also \emph{fci} and \emph{mci}, so \emph{rci} homomorphisms satisfy the properties that both of these do. Despite this, we will also give direct proofs for the main properties to showcase their simplicity.
	
	\begin{defn} \label{defRCI}
		A local homomorphism $ f \colon \locala \to \localb $ of noetherian local rings is said to be \emph{reasonably complete intersection} (\emph{rci} for short)  if there exists a regular factorization 
		\[ A \to R \to \widehat{B} \]
		and a surjective homomorphism of noetherian local rings $ Q \to R $ such that 
		\[ \HH_n(Q,R, l)=0=\HH_n(Q, \widehat{B}, l)  \quad \text{for all } n \geq 2.\]
	\end{defn}
	
	\begin{rem}
		If $ f $ is \emph{rci} then $ \HH_n( A, B, l)=0 $ for all $ n \geq 3 $. Indeed, $ \HH_n( A, B, l) = \HH_n( R, \widehat{B}, l) $ for all $ n \geq 2 $ (cf. \cite[Lemma 1.7]{AvramovAnnals}), and now use the Jacobi-Zariski exact sequence
		\[ \cdots \to \HH_n( Q, \widehat{B}, l) \to \HH_n( R, \widehat{B}, l) \to \HH_{n-1}( Q, R, l) \to \cdots \]
	\end{rem}
	
	\begin{prop} \label{prop2of3RCI}
		For a local homomorphism $ f \colon \locala \to \localb $ of noetherian local rings, any two of the following properties imply the third one:
		\begin{enumerate}
			\item[(i)] $ f $ is \emph{rci}.
			\item[(ii)] $A$ is a complete intersection ring.
			\item[(iii)] $B$ is a complete intersection ring.
		\end{enumerate} 
	\end{prop}
	\begin{proof}
		(ii) + (iii) $ \Rightarrow $ (i). Take a Cohen factorization $ A \to R \to \widehat{B} $ (in particular, $ R $ is complete). Since $ A $ is complete intersection and the closed fibre of $ A \to R $ is regular, $ R $ is also complete intersection. Taking a regular local ring $ Q $ such that $ R  = Q/I $, we get the desired result.
		
		(i) + (ii) $ \Rightarrow $ (iii). Let $ R $ and $ Q $ be as in Definition \ref{defRCI}. $ A $ is complete intersection, hence $ R $ is too, and therefore so is $ Q $ and, finally, so is $ \widehat{B} $.
		
		(i) + (iii) $ \Rightarrow $ (ii). Let $ R $ and $ Q $ be as in Definition \ref{defRCI}. From $ B $ being complete intersection and $ \HH_2(Q, \widehat{B}, l)=0 $, we deduce that $ Q $ is also complete intersection. Now, since $ \HH_2( Q, R, l)=0 $, so too is $ R $ complete intersection, and therefore also $ A $ by flat descent.
	\end{proof}
	
	\begin{rem}
		Proposition \ref{prop2of3RCI} would still hold should we change ``for all $ n \geq 2 $'' to ``for all $ n \geq 3 $'' in condition $ \HH_n( Q, \widehat{B}, l)=0 $ and/or $ \HH_n( Q, R, l)=0 $ in Definition \ref{defRCI}.
	\end{rem}
	
	\begin{prop}
		Let $ f \colon \locala \to \localb $ be a local homomorphism  of noetherian local rings. If $ f $ is \emph{rci}, then:
		\begin{enumerate}
			\item[(i)] $ A $ is Gorenstein if and only if $ B $ is Gorenstein.
			\item[(ii)] $ A $ is Cohen-Macaulay if and only if $ B $ is Cohen-Macaulay.
		\end{enumerate}
	\end{prop}
	\begin{proof}
		It is the same proof as in (i)+(ii) $\Rightarrow$ (iii) and (i)+(iii)$\Rightarrow$(ii) in Proposition \ref{prop2of3RCI}. Alternatively, the result follows from Proposition \ref{relationRCIMCIFCI} and Theorem \ref{ascentDescentMCI}.
	\end{proof}
	
	The following proposition summarizes what we know about the relationships between these concepts. Following \cite{AvramovAnnals} and \cite{AvramovHenriquesSega}, we will say that $ f \colon \locala \to \localb$ is \emph{ci} (complete intersection) if $ \HH_n( A, B, l)=0 $ for all $n \geq 2 $, and that $ f $ is \emph{qci} (quasi-complete intersection) if  $ \HH_n( A, B, l)=0 $ for all $n \geq 3 $. 
	
	\begin{prop} \label{relationRCIMCIFCI}
		Let $ f \colon \locala \to \localb $ be a local homomorphism  of noetherian local rings. Then, the properties that $ f $ can satisfy behave according to the following diagram:
		
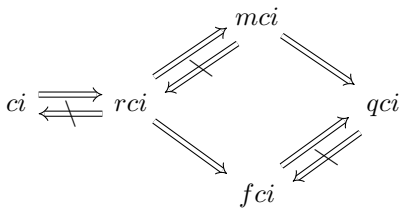
\begin{center}
			\begin{tikzcd}
				& & mci & 
				\\ci & rci & & qci
				\\ & & fci &  
				
				\arrow[from=2-1, to=2-2, Rightarrow, shift left=1.5]
				\arrow[from=2-2, to=2-1, negated, Rightarrow, shift left=1.5]
				\arrow[from=2-2, to=1-3, Rightarrow, shift left=1.5]
				\arrow[from=1-3, to=2-2, negated, Rightarrow, shift left=1.5]
				\arrow[from=2-2, to=3-3, Rightarrow]
				\arrow[from=1-3, to=2-4, Rightarrow]
				\arrow[from=3-3, to=2-4, Rightarrow, shift left=1.5]
				\arrow[from=2-4, to=3-3, negated, Rightarrow, shift left=1.5]
			\end{tikzcd}
		\end{center}
	\end{prop}
	\begin{proof}
		It remains to be shown that:
		\begin{enumerate}
			\item[(a)] \emph{ci} $ \Rightarrow $ \emph{rci}. If $ \HH_n( A, B, l)=0 $ for all $ n \geq 2 $, then by taking a Cohen factorization we have $ \HH_n( R, \widehat{B}, l)=0 $ for all $ n \geq 2 $. Hence, it suffices to take $ Q  \eqdef R$. 
			\item[(b)] \emph{rci} $ \nRightarrow $ \emph{ci}. If $ A $ and $ B $ are complete intersection, then $ f $ is \emph{rci}.
			\item[(c)] \emph{rci} $\Rightarrow$ \emph{fci}. It is easy to see that \emph{rci} implies condition (i) in Proposition \ref{KoszulFreeness} and then, by that same Proposition, \emph{rci} implies \emph{fci}. However, we can give a short and direct proof: let $ R $ and $ Q $ be as in Definition \ref{defRCI} and take $ A'=R $ in Definition \ref{defMCI}. We have $ \HH_n( Q, R, l)=0 $ for all $ n \geq 2 $, so by the Jacobi-Zariski exact sequence associated with $ Q \to R \to B \otimes_A R $
			\[ \cdots \to \HH_n( Q, R, l) \to \HH_n( Q, B \otimes_A R, l) \xrightarrow{\gamma_n} \HH_n( R, \widehat{B} \otimes_A R, l) \to \cdots \]
			we obtain $ \HH_n( Q, B \otimes_A R, l)=0 $ for all $ n \geq 2 $, because $ \gamma_n=0 $ for all $ n \geq 2 $ as shown by the naturality of André-Quillen homology with respect to
			\begin{center}
				\begin{tikzcd}
					Q & R & \widehat{B} \otimes_A R
					\\Q & R & \widehat{B}
					
					\arrow[from=1-1, to=1-2]
					\arrow[from=1-2, to=1-3]
					\arrow[from=1-3, to=2-3]
					\arrow[from=1-1, to=2-1, equal]
					\arrow[from=2-1, to=2-2]
					\arrow[from=2-2, to=2-3]
					\arrow[from=1-2, to=2-2, equal]
				\end{tikzcd}
			\end{center}
			(where the homomorphism $ \widehat{B} \otimes_A R \to \widehat{B} $ is multiplication), that is, the diagram
			\begin{center}
				\begin{tikzcd}
					& &  \HH_n( A, \widehat{B}, l)
					\\\HH_n( Q, \widehat{B} \otimes_A R, l) & \HH_n( R, \widehat{B} \otimes_A R, l) 
					\\ \HH_n( Q, \widehat{B}, l) & \HH_n( R, \widehat{B}, l)
					
					\arrow[from=2-1, to=2-2, "\gamma_n"]
					\arrow[from=2-2, to=3-2]
					\arrow[from=2-1, to=3-1]
					\arrow[from=3-1, to=3-2]
					\arrow[from=1-3, to=2-2, "\eta", swap]
					\arrow[from=1-3, to=3-2, "\pi"]
				\end{tikzcd}
			\end{center}
			In the diagram above, $ \HH_n( Q, \widehat{B}, l) =0 $ for $ n \geq 2 $, while $ \eta $ is an isomorphism for all $ n $ by (flat) base change and $\pi$ is an isomorphism for $ n \geq 2 $ due to \cite[Lemma 1.7]{AvramovAnnals}. Thus, by the commutativity of that triangle, we obtain $ \gamma_n=0 $ for all $ n \geq 2 $. 
			
			\item[(d)] \emph{rci} $\Rightarrow$ \emph{mci}. If $ f $ is \emph{rci}, then we have an exact sequence
			\begin{equation} \tag{1} \label{eq1}
				0 = \HH_2( Q, \widehat{B}, l) \to \HH_2( R, \widehat{B}, l) \to \HH_1( Q, R, l) \to \HH_1( Q, \widehat{B}, l) \to \HH_1( R, \widehat{B}, l) \to 0
			\end{equation}
			Now, $ \HH_2( Q, R, l)=0=\HH_2( Q, \widehat{B}, l)=0 $, so $ I = \ker(Q \to R) $ and $ J = \ker(Q \to \widehat{B}) $ are generated by regular sequences and thus
			\begin{equation} \tag{2} \label{eq2}
				\dim_l \HH_1( Q, R, l)= \dim_l I/\mathfrak{n}I = \mu(I) = \dim Q - \dim R
			\end{equation}
			\begin{equation} \tag{3} \label{eq3}
				\dim_l \HH_1( Q, \widehat{B}, l)= \dim_l J/\mathfrak{n}J = \mu(J) = \dim Q - \dim \widehat{B}
			\end{equation}
			(where $ \mathfrak{n} $ is the maximal ideal of $ Q $ and $\mu(-)$ denotes the minimal number of generators). From \eqref{eq1}, \eqref{eq2} and \eqref{eq3},
			\begin{align*}
				0 &= \dim_l \HH_2( R, \widehat{B}, l) - \dim_l \HH_1( R, \widehat{B}, l) + \dim_l \HH_1( Q, B, l) - \dim_l \HH_1( Q, R, l)=
				\\&= \dim_l \HH_2( R, \widehat{B}, l) - \dim_l \HH_1( R, \widehat{B}, l) + \dim R - \dim B = \mathrm{d}(f).
			\end{align*}
			
			\item[(e)] \emph{mci} $ \nRightarrow $ \emph{rci}. In \cite[Theorem 3.5]{AvramovHenriquesSega}, it is shown that there exists a noetherian local ring $ A $ of dimension $ 0 $ and an exact zero-divisor $ x \in A $ (that is, $ 0 \neq (0 \colon x)  $ is a free module over $ A/(x) $, and thus of rank $ 1 $; see \cite[Proposition 1]{SotoGlasgow}) such that $ \cidim_A(B) = \infty $, where $ B = A/(x) $. Let $ f \colon A \to B $ be the canonical homomorphism. Since $ \HH^{Kos}_1(x) = (0 \colon x) $ is a free $ B $-module of rank $ 1 $, $ \HH_n( A, B, -)=0 $ for all $ n \geq 3 $ by \cite[Corollary 3']{BlancoMajadasRodicioJPAA}. Therefore, by Lemma \ref{lemmaGradeDepth}, $ \mathrm{d}(f)=0 $ and thus it is \emph{mci}. However, $ f $ is not \emph{rci}, since this would mean by (c) that $ f $ is also \emph{fci}, but Remark \ref{counterExampleH3fci} already settled that this is not true.
		\end{enumerate}
	\end{proof}


\begin{thebibliography}{10}
		
		\bibitem{An1974}
		Michel Andr\'e, \emph{Homologie des Alg\`ebres Commutatives}. Springer, 1974.
		
		\bibitem{AnPairs}
		Michel Andr\'e, \emph{Pairs of complete intersections}. J. Pure Appl. Algebra {\bf 38} (1985), no.~2-3, 127--133;
		
		\bibitem{AvramovCI}
		Luchezar L. Avramov, \emph{Flat morphisms of complete intersections}. Dokl. Akad. Nauk SSSR {\bf 225} (1975), no.~1, 11--14.
		
		\bibitem{AvramovMathAnn}
		Luchezar L. Avramov, \emph{Homology of local flat extensions and complete intersection defects}. Math. Ann. {\bf 228} (1977), no.~1, 27--37.
		
		\bibitem{AvramovCRAS}
		Luchezar L. Avramov, \emph{Descente des d\'eviations par homomorphismes locaux et g\'en\'eration des id\'eaux de dimension projective finie}. C. R. Acad. Sci. Paris S\'er. I Math. {\bf 295} (1982), no.~12, 665--668.
		
		\bibitem{AvramovAnnals}
		Luchezar L. Avramov.  \emph{Locally complete intersection homomorphisms and a conjecture of Quillen on the vanishing of cotangent homology}.
		Ann. of Math. (2)  150  (1999),  no. 2, 455--487.
		
		\bibitem{AvramovHenriquesSega}
		Luchezar L. Avramov, I.~B.~D.~A. Henriques, Liana M. \c Sega, \emph{Quasi-complete intersection homomorphisms}. Pure Appl. Math. Q. {\bf 9} (2013), no.~4, 579--612.
		
		\bibitem{AvramovFoxbyHalperin}
		Luchezar L. Avramov, Hans-Bjørn Foxby, Stephen Halperin, \emph{Descent and ascent of local properties along homomorphisms of finite flat dimension}. J. Pure Appl. Algebra {\bf 38} (1985), no.~2-3, 167--185.
		
		\bibitem{AvramovFoxbyHerzog}
		Luchezar L. Avramov, Hans-Bjørn Foxby, Bernd Herzog. \emph{Structure of local homomorphisms}.
		J. Algebra  164  (1994),  no. 1, 124--145.
		
		\bibitem{AvramovGasharovPeeva}
		Luchezar L. Avramov, Vesselin N. Gasharov and Irena V. Peeva, \emph{Complete intersection dimension}, Inst. Hautes \'Etudes Sci. Publ. Math. No. 86 (1997), 67--114 (1998).
		
		\bibitem{BlancoMajadasRodicioInventiones}
		Amalia Blanco, Javier Majadas, Antonio G. Rodicio, \emph{Projective exterior Koszul homology and decomposition of the Tor functor}. Invent. Math. {\bf 123} (1996), no.~1, 123--140.
		
		\bibitem{BlancoMajadasRodicioJPAA}
		Amalia Blanco, Javier Majadas, Antonio G. Rodicio, \emph{On the acyclicity of the Tate complex}, J. Pure Appl. Algebra {\bf 131} (1998), no.~2, 125--132.
		
		\bibitem{BouchibaCondeMajadas}
		Samir Bouchiba, Jesús Conde-Lago, Javier Majadas, \emph{Cohen-Macaulay, Gorenstein, complete intersection and regular defect for the tensor product of algebras}. J. Pure Appl. Algebra {\bf 222} (2018), no.~8, 2257--2266.
		
		\bibitem{BriggsIyengar}
		Benjamin Briggs, Srikanth B. Iyengar, \emph{Rigidity properties of the cotangent complex}. J. Amer. Math. Soc. {\bf 36} (2023), no.~1, 291--310.
		
		\bibitem{BruHer}
		Winfried Bruns, Jürgen Herzog, \emph{Cohen-Macaulay rings}. Cambridge Studies in Advanced Mathematics, 39, Cambridge Univ. Press, Cambridge, 1993.
		
		\bibitem{EGAISpri} Alexander Grothendieck, Jean Dieudonné. \emph{\'El\'ements de g\'eom\'etrie alg\'ebrique. I}. Springer, 1971.
		
		\bibitem{EGAIV1}
		Alexander Grothendieck, Jean Dieudonné. \emph{\'El\'ements de g\'eom\'etrie alg\'ebrique. IV. \'Etude locale des sch\'emas et des morphismes de sch\'emas I}. Inst. Hautes \'Etudes Sci. Publ. Math. No. 20 (1964).
		
		\bibitem{GarciaSoto}
		Antonio García R., José J. M. Soto, \emph{Ascent and descent of Gorenstein property}. Glasg. Math. J. {\bf 46} (2004), no.~1, 205--210.
		
		\bibitem{KiehlKunz}
		Reinhardt Kiehl, Ernst Kunz, \emph{Vollst\"andige Durchschnitte und $p$-Basen}. Arch. Math. {\bf 16} (1965), 348--362.
		
		\bibitem{MajadasNagoya}
		Javier Majadas, \emph{A descent theorem for formal smoothness}. Nagoya Math. J. {\bf 229} (2018), 113--140.
		
		\bibitem{Quillen}
		Daniel G. Quillen, \emph{On the (co-) homology of commutative rings}, in Applications of Categorical Algebra, Proc. Sympos. Pure Math,
		vol. XVII, pp. 65–87, Am. Math. Soc., Providence, 1970.
		
		\bibitem{RodicioJPAA}
		Antonio ~G. Rodicio, \emph{On the free character of the first Koszul homology module}. J. Pure Appl. Algebra {\bf 80} (1992), no.~1, 59--64.
		
		\bibitem{RodicioHelv}
		Antonio G. Rodicio, \emph{Flat exterior Tor algebras and cotangent complexes}. Comment. Math. Helv. {\bf 70} (1995), no.~4, 546--557. Erratum: Comment. Math. Helv. {\bf 71} (1996), no.~2, 338.
		
		\bibitem{Soto}
		José J. M. Soto, \emph{Finite complete intersection dimension and vanishing of Andr\'e-Quillen homology}. J. Pure Appl. Algebra {\bf 146} (2000), no.~2, 197--207.
		
		\bibitem{SotoGlasgow}
		José J. M. Soto, \emph{Gorenstein quotients by principal ideals of free Koszul homology}. Glasg. Math. J. {\bf 42} (2000), no.~1, 51--54.
		
		
		
	\end{thebibliography}
\end{document}